\documentclass{amsart}
\usepackage{amsmath, amssymb, amsthm, amsfonts}
\usepackage{hyperref}
\usepackage{cleveref}

\usepackage{blkarray}
\usepackage{enumerate}

\usepackage{graphicx}
\usepackage[all]{xy} 
\usepackage{mathrsfs}
\usepackage{eufrak}

\theoremstyle{plain}
\newtheorem{thm}{Theorem}[section]
\newtheorem{lem}[thm]{Lemma}
\newtheorem{prop}[thm]{Proposition}
\newtheorem{cor}[thm]{Corollary}

\theoremstyle{definition}
\newtheorem{defn}[thm]{Definition}

\theoremstyle{remark}
\newtheorem{rem}[thm]{Remark}
\newtheorem{Ex}[thm]{Example}

\begin{document}

\title{Degree of the 3-secant variety}

\author{Doyoung Choi} 
\date{\today}

\address{Doyoung Choi \\
Department of Mathematical Sciences, Korea Advanced Institute of Science and Technology (KAIST), 373-1 Gusung-dong, Yusung-Gu, Daejeon, Republic of Korea}
\email{cdy4019@kaist.ac.kr}

\begin{abstract} 
    In this paper, we present a formula for the degree of the 3-secant variety of a nonsingular projective variety embedded by a 5-very ample line bundle. The formula is provided in terms of Segre classes of the tangent bundle of a given variety. 
    We use the generalized version of the double point formula to reduce the calculation into the case of the 2-secant variety. As a resolution of the 2-secant variety, we use secant bundle and compute powers of desired algebraic cycles.
\end{abstract}

\keywords{double point formula, higher very ampleness, secant bundle, Segre class, secant variety}
\subjclass[2020]{Primary:~14N07; Secondary:~14C25,~14N10}

\maketitle
\section{Introduction}
  Let $X$ be a projective variety over an algebraically closed field of characteristic zero. The \textit{$k$-secant variety}, denoted $\sigma_k(X)$, is defined as the Zariski closure of the union of $k$-secant hyperplanes $\mathbb{P}^{k-1}$ to $X$ in projective space $\mathbb{P}^N$. For example, the 3-secant variety, denoted $\sigma_3(X)$, is the Zariski closure of the union of all secant planes to $X$ in $\mathbb{P}^N$. The study of secant varieties is a classic topic in algebraic geometry with an emphasis on determining defining equations and syzygies as well as understanding the singularities of these varieties.

  One traditional topic of study is the specification of the double points of linear projections of a given variety. Let $X \subset \mathbb{P}^N$ be an $n$-dimensional nonsingular projective variety. Let $\delta$ be the number of double point of the image of $X$ under the generic linear projection $X \to \mathbb{P}^{2n}$. In \cite[Corollary 8.2.9]{joins1999} it is shown that $\delta$ can be represented in terms of the Segre classes of the tangent bundle $ T_X$ as:
 $$ 2\delta = (\deg X)^2 - \sum_{ k \geq 0} {2n+1 \choose n-k} \deg s_k(T_X).   $$ 
   Furthermore, if the embedding of $X$ into the projective space is 3-very ample, the double point formula provides a degree formula for the 2-secant variety $\sigma_2(X)$ for $X$. 
The double point formula is one of the corollary of \cite[Theorem 8.2.8]{joins1999} and more generally, \cite[Theorem 2.1.15]{joins1999}. The following theorem is called as the refined Bezout's theorem:

\begin{thm} [{\cite[Theorem 2.1.15]{joins1999}}]
 Let $V \subset \mathbb{P}^N$ be an equi-dimensional closed subscheme with $\mathcal{L}:= \mathcal{O}_V(1)$. Let $\sigma_0 , \cdots , \sigma_d$ be global sections of $\mathcal{L}$. Let $v^i(\underline{\sigma}, V)$ be the $v$-cycle in the Vogel's intersection theory and $\sigma : V \dashrightarrow \mathbb{P}^d$ be a rational map defined by global section $\sigma_0, \cdots , \sigma_d$. Denote $\deg (\Gamma/\sigma(\Gamma))$ be the degree of the restriction of $\sigma$ on $\Gamma$ where $\Gamma$ is an irreducible component of $V$. 
   \begin{equation*}
      \deg V = \sum_i \deg v^i(\underline{\sigma}, V) + \sum_{\Gamma \subset V} \deg (\Gamma/\sigma(\Gamma)) \deg \sigma(\Gamma). 
   \end{equation*}  
   \end{thm}
 We can represent Vogel's $v$-cycles in terms of the first Chern class $c_1(\mathcal{L})$ and the total Segre class of a normal cone $C_WV$ by \cite[Theorem 2.4.2, Corollary 2.4.7]{joins1999}.

  As a generalization of previous results on the degree of the 3-secant variety $\sigma_3(X)$ for arbitrary dimension of $X$, the main theorem is presented as follows: 
 
 \begin{thm}

    Let $X$ be a nonsingular projective variety embedded by a 5-very ample line bundle. Let $n$ and $d$ be the dimension and  the degree of $X$. Then the degree of the 3-secant variety $\sigma_3(X)$ is given by the following formula :
$$ \frac{1}{3!} \left( d^3 - \sum_{k=0}^{n} d \; a_{n,k} \;  \deg s_k(T_X) + \sum_{k=0}^{n} \sum_{a=0}^{k} 2^{k-a+n+1} {3n+2 \choose n-k} \; \deg(s_a(T_X) \cdot s_{k-a}(T_X))\right)  $$
where $a_{n,k} = {2n+1 \choose n-k} + 2\sum_{i=k}^{n} (-1)^{i-k} {3n+2 \choose n-i}{ i-k+n \choose n}. $
\end{thm}

  When attempting to compute the degree of the 3-secant variety for arbitrary dimension of $X$, it is not possible to imitate the approach in \cite{lehn} by Lehn. It is because the universal family $Z_3 \subset X \times X^{[3]}$ may be singular(cf. \cite{GF93} ). Additionally the refined Bezout's theorem cannot be applied directly. One possible approach is to set $V$ as the triple join $J(X, X, X)$ of $X$ and $\sigma$ as the addition map defined by $$[x_0, \cdots x_N,y_0, \cdots, y_N,z_0, \cdots z_N] \mapsto [x_0+y_0+z_0, \cdots , x_N + y_N + z_N].$$
  
  The image $\sigma(V)$ is the 3-secant variety $\sigma_3(X)$. However the indeterminacy locus of $\sigma$ can be non-reduced, making computation of $v^i(\underline{\sigma}, V)$ much more difficult. Therefore it is not advisable to imitate the approach in \cite{joins1999}. 

  To circumvent these issues, we propose to use the secant bundle, which was introduced by Schwarzenberger in \cite{Schw64} for the purpose of giving fiber bundle structure of secant lines. While the original construction was based on symmetric product of a given variety, rather than the Hilbert scheme of points, in this paper we adopt the convention presented in \cite{bertram1992moduli} and \cite{vermeire2009singularities} where the secant bundle is a projective bundle over a Hilbert scheme of points. This approach has been used in previous works such as \cite{bertram1992moduli},  \cite{ullery}, \cite{vermeire2001} and \cite{vermeire2009singularities}, where the secant bundles was used as a tool for describing the singularity and the normality of the 2-secant variety. In this paper, we utilize the birational morphism from the secant bundle to the 2-secant variety as a resolution of singularities.

In Section 2, we define the total Segre class of a cone and introduce a  generalized version of the double point formula. This  formula allows us to compute the total Segre class $s(X, \sigma_2(X))$ in order to determine the degree of the 3-secant variety. However, singularities of the 2-secant variety may impede this computation. To overcome this issue, we introduce the notion of higher very ampleness of a line bundle $\mathcal{L}$ and the secant bundle in Section 3. We take the secant bundle  as a nonsingular birational model for the 2-secant variety when the line bundle $\mathcal{L}$ satisfies the higher very ampleness condition. As shown in \cite{ullery}, the inverse image of $X$ under this birational morphism is isomorphic to the universal family $Z_2 \subset X \times X^{[2]}$ when $\mathcal{L}$ is 3-very ample. Since the codimension of $Z_2$ in the secant bundle is 1, we can treat $Z_2$ as an effective divisor, which helps in computation. In Section 4, we derive the main theorem which provides the degree formula for the 3-secant variety by using the refined Bezout's theorem and the total Segre class computed in Section 3. Additionally, we compute the multiplicity of $\sigma_2(X)$ along $X$ as a corollary to main theorem which is already known in \cite[Proposition 6.1]{ChouSong18} . 
In Section 5, we apply the main theorem to cases of curves and surfaces with explicit calculations. \\

\noindent \textbf{Acknowledgments.} This research was partially supported by the Institute for Basic Science (IBS-R032-D1-2022-a00) and the author would like to thank to Yongnam Lee for his suggestions for research topics. I also would like to thank to Haesong Seo and Chiwon Yoon for suggesting the notation of points of Hilbert scheme of points.

\section{Generalized double point formula}
  The generalized version of the double point formula is one of a useful application of the refined Bezout's theorem in intersection theory. The double point formula allows us to compute the degree of the intersection of two subvarieties of a projective space. In this paper, we will use it as a starting point to understand the complexity of the 3-secant variety in terms of the 2-secant variety.

To set the stage for the generalized version of double point formula, we first introduce some notation and conventions. Let $X$ be an algebraic scheme over a field $k$ and $S^{\cdot} := \bigoplus_{\nu} S^{\nu}$ be a sheaf of graded $\mathcal{O}_X$-algebras. We assume that the map $\mathcal{O}_X \to S^{0}$ is surjective, $S^1$ is coherent, and $S^{\cdot}$ is generated by $S^1$. We define the cone of $S^{\cdot}$ as the relative spectrum $C:= \mathbf{Spec}(S^{\cdot})$ and the projection $C \to X$. The projective cone of $S^{\cdot}$ is defined as the relative projective spectrum $P(C) := \mathbf{Proj}(S^{\cdot})$ with the projection $p:P(C) \to X$. We also define the projective completion of $C$ by $P(C \oplus 1) := \textbf{Proj}( S^{\cdot}[z])$ where $z$ is an indeterminate, with projection $q : P(C \oplus 1) \to X$, and let $\mathcal{O}(1)$ be the tautological line bundle on $P(C\oplus 1)$.
 \begin{rem} 
  Throughout this paper, we use the convention for projective bundle as in \cite[Appendix B.5.5]{fulton2013}.
  \end{rem}

 Let $A_k(X)$ be the Chow group of dimension $k$ algebraic cycles of $X$, and let $A^k(X)$ be the Chow group of codimension $k$ algebraic cycles of $X$. The notation $A_*(X)$ stands for the direct sum of $A_k(X)$ for each $k \geq 0$, considering algebraic cycles in terms of their dimensions. Similarly, $A^*(X)$ stands for the direct sum of $A^k(X)$ for each $k \geq 0$, considering algebraic cycles in terms of their codimensions. 
 With these conventions in place, we can now define the Segre class of a cone.

 \begin{defn}[{\cite[Chapther 4]{fulton2013}}]
 For a variety $V$, we denote the algebraic cycle corresponding to $V$ as $[V]$.  
     The \textit{Segre class} of $C$ is the class in $A_*(X)$ defined by 
  $$s(C) := q_*(\sum_{i \geq 0} c_1(\mathcal{O}(1))^i \cap [P(C\oplus 1)])$$
  or equivalently, $$s(C) = p_*(\sum_{i \geq 0} c_1(\mathcal{O}(1))^{i-1} \cap [P(C)]).$$
 \end{defn}

  The Segre class associated with the normal cone of a closed immersion $X \hookrightarrow Y$ is denoted as $s(C_XY)$ or simply $s(X, Y)$. The dimension $i$ component of the Segre class $s(C)$ is represented by $s_i(C)$, and for the specific case of $s(C_XY)$, we use the notation $s_i(C_XY) $. \\

   Consider a smooth algebraic variety $X$ of dimension $n$, and let $\Delta(X) \subset X \times X$ denote the diagonal embedding. The normal cone $C_{\Delta(X)}(X \times X)$ is canonically isomorphic to the geometric vector bundle associated to the tangent sheaf $T_{X}$. For a locally free sheaf $T_X$, we adopt the codimension convention with respect to X for lower indices of Chern classes and Segre classes, denoted as  $s_k(T_X)$. (cf. \cite[Chapter 3.1]{fulton2013})

    However, when considering the normal cone $C_{\Delta(X)}(X \times X)$, a different convention is used for the dimension of the lower indices of Segre classes. In this case, it is denoted as $s_{n-k}(C_{\Delta(X)}(X \times X))$ (cf. \cite[Chapter 4.1]{fulton2013}). The relationship between these two conventions is established as follows:

 \begin{equation} \label{segre tangent bd normal cone}
 s_k(T_X)  = s_{n-k}(C_{\Delta(X)}(X \times X)).
\end{equation}

For the better understanding of the generalized version of the double point formula, we give a definition of two kind of join varieties. 
\begin{defn} [{\cite[Chapter 1.3]{joins1999}}]
Let $X$ and $Y$ be subvarieties of $\mathbb{P}^N$. The \textit{embedded join} of $X$ and $Y$ is the closure of the union of all lines connecting a point in $X$ to a point in $Y$, denoted $XY$. 
\end{defn}

\begin{defn}
 Let $X$ be a subvariety of $\mathbb{P}^N$ and $Y$ be a subvariety of $\mathbb{P}^M$. The \textit{abstract} or \textit{ruled join} $J(X,Y)$ is the set of points $[x_0 : \cdots : x_N : y_0 : \cdots : y_M]$ of $\mathbb{P}^{N+M+1}$ where $[x_0 : \cdots : x_N]$ is a point of $X$ and $[y_0 : \cdots : y_M]$ is a point of $Y$. Simply, we define $J(X,Y)$ as $$ \{ [x:y] \in \mathbb{P}^{N+M+1} | x \in \hat{X} \; \textrm{and} \; y \in \hat{Y} \} $$ where $\hat{X}$ and $\hat{Y}$ are affine cones of $X$ and $Y$, respectively. 
\end{defn}

  Let $X$ and $Y$ be two projective subvarieties of the projective space $\mathbb{P}^N$, and let $x$ and $y$ be closed points of $X$ and $Y$, respectively, with coordinate representations $[x_0: \cdots : x_N]$ and $[y_0: \cdots : y_N]$. We define a rational map $J(X, Y) \dashrightarrow XY$ as a linear projection such that 
$$[x_0: \cdots : x_N: y_0: \cdots : y_N] \mapsto [x_0-y_0: \cdots : x_N-y_N].$$
The indeterminate locus of this map is defined by the equations $$x_0-y_0 = \cdots = x_N-y_N = 0.$$ Since the map $J(XY) \dashrightarrow XY$ is dominant by \cite[p.26]{joins1999}, we can define the degree of the rational map $J(X,Y) \dashrightarrow XY$. Denote the degree by $\textrm{deg}(J/XY)$.

The intersection $X \cap Y$ is embedded into the product variety $X \times Y$ along its diagonal embedding. Let $C_{X \cap Y}(X \times Y)$ denote the normal cone of $X \cap Y$ to $X \times Y$. The following theorem decomposes the information of the embedded join into simpler pieces:

\begin{thm}  [{\cite[Theorem 8.2.8]{joins1999}}] For subvarieties $X, Y$ of $\mathbb{P}^N$ of degree $d_X,$ resp. $d_Y$ and dimension $n, $ resp. $m$ we have
   \begin{equation*}
       \deg XY  \deg(J/XY) = d_Xd_Y \; - \sum_{k \geq 0} {n+m+1 \choose k} \; \deg s_k(C_{X\cap Y}(X \times Y)). 
   \end{equation*}  \end{thm}
    This formula is derived from the refined Bezout's theorem, which is based on Vogel's $v$ and $\beta$ cycle construction. For further explanations, see \cite[Chapter 2]{joins1999} and \cite[Chapter 8]{joins1999}.

 If we put $Y = X$, then the embedded join $XY$ is the 2-secant variety $\sigma_2(X)$.  Assume that $\mathcal{O}_X(1)$ is a 3-very ample line bundle. By \cite[Theorem 8.2.8]{joins1999}, we get a degree formula for the 2-secant variety :
\begin{equation} \label{eq1}
\deg \sigma_2(X) = \frac{1}{2} \left( d_X^2 - \sum_{k \geq 0} {2n+1 \choose k} \; \deg s_k(C_{\Delta(X)}(X \times X)) \right) .
\end{equation}

If we set $Y = \sigma_2(X)$, then the embedded join $XY$ is the 3-secant variety $\sigma_3(X)$. It seems good to apply \cite[Theorem 8.2.8]{joins1999} for the degree formula of the 3-secant variety. In this case, $d_Y$ is the degree of the 2-secant variety $\sigma_2(X)$. In order to compute, we need the degrees of the Segre classes $s_k(C_{\Delta(X)}(X \times \sigma_2(X)))$.

If $\mathcal{O}_X(1)$ is 5-very ample, the rational map $J(X,Y) \dashrightarrow XY$ has degree 3. (We will prove it at Lemma \ref{lem}) Then we obtain the degree formula for the 3-secant variety $\sigma_3(X)$ as follows : 
\begin{equation} \label{eq2}
\frac{1}{3} \left( \frac{d_X \; d_Y}{2}  - A  \right)
\end{equation}
where 
\begin{equation*}
 d_Y = \deg \sigma_2(X)
\end{equation*}  and
\begin{equation} \label{eq3}
A:= \sum_{k \geq 0} {3n + 2 \choose k} \deg s_k(C_{\Delta(X)} (X \times \sigma_2(X))).
\end{equation}

    The equation (\ref{eq1}) is just a linear combination of Segre classes of tangent bundle of $X$. 
    Therefore, it is sufficient to calculate the term $s(C_{\Delta(X)}(X \times \sigma_2(X)))$ in order to determine the degree of the 3-secant variety. However, the normal cone $C_{\Delta(X)} (X \times \sigma_2(X))$ has singularities along the diagonal $\Delta(X)$, which must be expressed in simpler forms for the calculation to proceed.

   \section{The secant bundle}

In this section, we will introduce and utilize the concept of the secant bundle as a nonsingular birational model for the 2-secant variety. The secant bundle is a projective bundle over the Hilbert scheme $X^{[2]}$, which represent length 2 subschemes of $X$. The fiber of the projection morphism of the secant bundle at each point $\xi$ of $X^{[2]}$ corresponds to either a secant line or a tangent line spanned by $\xi$. Under sufficient high very ampleness of embedding line bundle, the distinct secant and tangent lines do not intersect in the secant bundle.

\begin{defn} [cf. {\cite{BS88} and \cite{catanese90}}]
    A line bundle $\mathcal{L}$ on a complete algebraic variety $X$ over an algebraically closed field $k$ is \textit{d-very ample} if, for every zero-dimensional subscheme $Z$ of $X$ with length less than or equal to $d+1$, the restriction map $$r_Z: H^0(X, \mathcal{L}) \to H^0(X, \mathcal{L} \otimes \mathcal{O}_Z)$$ is surjective.
\end{defn}

\begin{rem}
 Note that for two integers $d_1 \geq d_2$, if a line bundle $\mathcal{L}$ is $d_1$-very ample, $\mathcal{L}$ is also $d_2$-very ample.
\end{rem}

A line bundle $\mathcal{L}$ is 0-very ample if and only if it is spanned by global sections, and it is 1-very ample if and only if it is very ample. For instance, the $d$-uple Veronese embedding of $\mathbb{P}^n$ is embedded by the $d$-very ample line bundle $\mathcal{O}(d)$, as shown in \cite[Corollary 2.1 and Proposition 2.2]{BS93}.

   Let $X^{[2]}$ be the Hilbert scheme of 2 points of $X$ and $\mathbb{P}^N$ be the projective space $\mathbb{P}(H^0(X, \mathcal{L}))$. It is known that $X^{[2]}$ is smooth for any dimension of $X$(cf. \cite{GF93}). We consider the projections $\pi_1 : X \times X^{[2]} \to X$ and $\pi_2 : X \times X^{[2]} \to X^{[2]}$, and the universal family $Z_2 \subset X \times X^{[2]}$. We let $I_{Z_2}$ be the ideal sheaf of $Z_2$ on $X \times X^{[2]}$ and $\mathcal{O}_{Z_2}$ be the structure sheaf of $Z_2$.

 We construct the following exact sequence of sheaves:
$$ 0 \to \pi_1^*\mathcal{L} \otimes I_{Z_2} \to \pi_1^* \mathcal{L} \to \pi_1^*\mathcal{L} \otimes \mathcal{O}_{Z_2} \to 0.$$

Since the restriction of $\pi_2$ to $Z_2$, $\pi_2|_{Z_2} : Z_2 \to X^{[2]}$ is flat of degree 2, the sheaf $\pi_{2 \; *}(\pi_1^* \mathcal{L} \otimes \mathcal{O}_{Z_2})$ is locally free of rank 2. Denote the tautological bundle associated with $\mathcal{L}$ by 
$$E_{2,\mathcal{L}} = \pi_{2\; *}(\pi_1^* \mathcal{L} \otimes \mathcal{O}_{Z_2}).$$
   The fiber of the vector bundle $E_{2,\mathcal{L}}$ at a point $\xi \in X^{[2]}$ is given by $H^0(X, \mathcal{L} \otimes \mathcal{O}_{\xi})$ by Grauert's theorem.  
    Since $\mathcal{L}$ is 1-very ample, the restriction map $$H^0(X, \mathcal{L}) \to H^0(X, \mathcal{L} \otimes \mathcal{O}_{\xi})$$ is surjective. This implies that the morphism $\pi_{2*} \pi_1^* \mathcal{L} \to E_{2,\mathcal{L}}$ is also surjective. As a result, we have a composition of morphisms
$$\mathbb{P}(E_{2,\mathcal{L}}) \to \mathbb{P}(H^0(X, \mathcal{L})) \times X^{[2]} \to \mathbb{P}(H^0(X,\mathcal{L})) \cong \mathbb{P}^N.$$
 Let $\pi : \mathbb{P}(E_{2, \mathcal{L}}) \to X^{[2]}$ be the projection map of projective bundle. In \cite{vermeire2009singularities}, it is shown that if $\mathcal{L}$ is a 1-very ample line bundle, the image of this map is the 2-secant variety $\sigma_2(X)$. We denote this map as $$r : \mathbb{P}(E_{2,\mathcal{L}}) \to \sigma_2(X).$$ The projective bundle $\mathbb{P}(E_{2,\mathcal{L}})$ is known as the secant bundle of lines, and it is birational to the 2-secant variety $\sigma_2(X)$ if $\mathcal{L}$ is 3-very ample (as shown in \cite{bertram1992moduli} and \cite{vermeire2009singularities}).

It is a well-known fact that the universal family $Z_2$ is isomorphic to the blow-up of $X \times X$ along its diagonal $\Delta(X)$. See \cite[Remark 2.5.4]{gottsche2006hilbert}. We denote the blow-up morphism by $$\eta: Bl_{\Delta(X)} (X \times X) \to X \times X$$ and the involution map by $$\rho: Bl_{\Delta(X)} (X \times X) \cong Z_2 \to X^{[2]}.$$

The exceptional divisor of $\eta$ on $Bl_{\Delta(X)} (X \times X)$ is denoted by $E$. The projections $X \times X \to X$ are denoted by $\textrm{pr}_i$. The following diagram commutes:

\begin{equation*}
\xymatrix@=20pt{
    Z_2 \ar[r]^{\eta} \ar[d]^{\rho} & X \times X \ar[d] \\
    X^{[2]} \ar[r]^{\epsilon} & X^{(2)} }
\end{equation*}

where $X^{(2)}$ is the quotient $(X \times X)/ S_2$ and $\epsilon: X^{[2]} \to X^{(2)}$ is the Hilbert-Chow morphism.

     In the argument above \cite[Lemma 1.2]{ullery}, it is shown that the set theoretic inverse image $r^{-1}(X)$ under the map $$r:\mathbb{P}(E_{2,\mathcal{L}}) \to \sigma_2(X)$$ is isomorphic to $Z_2$ when $ \mathcal{L}$ is 3-very ample. We can also prove that this is true for the scheme theoretic inverse image. At first, we observe that scheme theoretic fibers of $r$ at each closed points of $X$ is reduced. For that, we need to compute on sheaves of graded algebras. 
     
     \begin{prop}
   Let $X$ be a nonsingular projective variety of dimension $n$ embedded by a 3-very ample line bundle $\mathcal{L}$. For each closed point $x$ of $X$, the scheme-theoretic fiber $r^{-1}(x)$ is isomorphic to the blow up of $X$ along $x$.
     \end{prop}

\begin{proof}
 Consider a closed point $x$ of $X$, and identify the following isomorphisms : 
   \begin{align*}
 X^{[2]} \times \{x\} & \approx \textbf{Proj} \; \textrm{Sym}^{\cdot} H^0(X, \mathcal{L}|_{x}) \otimes_{k} \mathcal{O}_{X^{[2]}}\\
   X^{[2]} \times \mathbb{P}^N & \approx \textbf{Proj} \; \textrm{Sym}^{\cdot} H^0(X, \mathcal{L}) \otimes_{k} \mathcal{O}_{X^{[2]}}\\
   \mathbb{P}(E_{2, \mathcal{L}}) & \approx \textbf{Proj} \; \textrm{Sym}^{\cdot} E_{2, \mathcal{L}}\\
 \end{align*}

  Express the scheme-theoretic fiber $r^{-1}(x)$ as a scheme-theoretic intersection of $\mathbb{P}(E_{2, \mathcal{L}})$ with $X^{[2]} \times \{x\}$. Denote $S$ as $ \textrm{Sym}^{\cdot} H^0(X, \mathcal{L}) \otimes_{k} \mathcal{O}_{X^{[2]}}$. Also, for coherent sheaves of ideals $I$ and $J$ of $S$, denote $ \textrm{Sym}^{\cdot} H^0(X, \mathcal{L}|_{x}) \otimes_{k} \mathcal{O}_{X^{[2]}}$ and $ \textrm{Sym}^{\cdot} E_{2, \mathcal{L}}$ by $S/I$ and $S/J$, respectively. Let $I^l. J^l$ and $S^l$ be degree $l$ part of $I, J$ and $S$. Then the scheme-theoretic intersection $r^{-1}(x)$ is obtained by $\textbf{Proj} \; S^{\cdot}/(I^{\cdot} + J^{\cdot}) $.
 
 Denote $F^l$ by $S^l/(I^l + J^l)$. The following diagram is a push-out diagram. 
\begin{equation*}
\xymatrix@=20pt{
    \textrm{Sym}^l H^0(X, \mathcal{L}) \otimes_{k} \mathcal{O}_{X^{[2]}} \ar[r]^{f^l} \ar[d]^{g^l} & \textrm{Sym}^l H^0(X, \mathcal{L}|_{x}) \otimes_{k} \mathcal{O}_{X^{[2]}} \ar[d] \\
    \textrm{Sym}^l E_{2,\mathcal{L}} \ar[r] & F^l }
\end{equation*}
 Hence we have 
$$ F^l \approx \left( \textrm{Sym}^l H^0(X, \mathcal{L}|_{x}) \otimes_{k} \mathcal{O}_{X^{[2]}} \oplus \textrm{Sym}^l E_{2, \mathcal{L}} \right) / N^l,$$
where $N^l$ is the image of morphism $(f^l, -g^l)$. 

 Let $[\xi]$ be a closed point of $X^{[2]}$ representing a zero-dimensional subscheme $\xi \subset X$, and denote $N^l|_{[\xi]}$ as a fiber of $N^l$. Then the fiber $F^l|_{[\xi]}$ is isomorphic to $\left( \textrm{Sym}^l H^0(X, \mathcal{L}|_{x}) \oplus \textrm{Sym}^l H^0(X, \mathcal{L}|_{\xi}) \right)/N^l|_{[\xi]}$.  
 Since $\textbf{Proj}\; F^l$ is isomorphic to $r^{-1}(x)$ and its support is homeomorphic to the blow up of $X$ along $x$, it suffices to show the vector space dimension of $F^l|_{[\xi]}$ is 1. Consider only when $\xi$ contains $x \in X$.

  As $g^1|_{[\xi]} : H^0(X, \mathcal{L}) \to H^0(X, \mathcal{L}|_{\xi})$ is surjective, its symmetrization $g^l$ is also surjective. Thus the rank of $F^l|_{[\xi]}$ is at most one. To prove the rank of $F^l|_{[\xi]}$ is at least 1, show that there is no section $s$ of $\textrm{Sym}^l H^0(X, \mathcal{L})$ such that $g^l(s) = s|_{\xi}$ is zero but $f^l(s) = s|_{x}$ is nonzero. For $l=1$, this is impossible since $x$ lies on $\xi$. For $l > 1$, consider the following : 
 $$ \textrm{Sym}^l H^0(X, \mathcal{L}) \approx \bigoplus_{a+b = l} \textrm{Sym}^a \ker g^1 \otimes_{k} \textrm{Sym}^b H^0(X, \mathcal{L}|_{\xi} ).$$

  As $g^l$ is surjective, we deduce that $\ker g^l$ is isomorphic to $$\bigoplus_{a+b = l, b<l} \textrm{Sym}^a \ker g^1 \otimes_{k} \textrm{Sym}^b H^0(X, \mathcal{L}|_{\xi} ).$$
This implies that if $g^l(s)$ is zero, then $f^l(s) = s|_{x}$ is also zero. Therefore, $F^l|_{[\xi]}$ is isomorphic to $\textrm{Sym}^l F^1|_{[\xi]}$ and $r^{-1}(x)$ is isomorphic to a relative projective spectrum of line bundle on the blow up of $X$ along $x$. Thus we prove this proposition.
\end{proof}

 Now we can establish that the scheme-theoretic inverse image $r^{-1}(X)$ is isomorphic to $Z_2$. Each fiber of $r^{-1}(X) \to X$ is smooth, satisfying conditions for an algebraic family of normal projective varieties. For every closed point $x$ of $X$, there exists a smooth curve $T$ passing through $x$. Consequently, $r^{-1}(T) \to T$ forms a flat family.  
  It is important to note that we can define an ample line bundle $\mathcal{M}$ on $r^{-1}(X)$ with respect to $r$,  given that $r^{-1}(X)$ is embedded in $X^{[2]} \times \mathbb{P}^N$. Subsequently the Hilbert polynomial of $\mathcal{M}$ on each fiber $r^{-1}(y)$ remains constant
 for closed points $y$ of $T$. Due to the connectedness of $X$, we can infer that $r^{-1}(X) \to X$ forms a flat family. So the morphism $r^{-1}(X) \to X$ is a smooth and hence $r^{-1}(X)$ is a nonsingular variety that is also reduced.\\

     Note that $X \times X^{[2]}$ is a closed subvariety of  $\mathbb{P}^N \times X^{[2]}$ and hence we can regard $Z_2$ is a closed subvariety of the secant bundle $\mathbb{P}(E_{2,\mathcal{L}})$ in a natural way. By adjusting the isomorphism, we can ensure that the composition of the maps $Z_2 \to X \times X^{[2]} \to X$ corresponds to the composition of maps $q:=\textrm{pr}_1 \circ \eta$. From this point on, we will identify $r^{-1}(X)$ with $Z_2$.

    Denote by $\Gamma_q$ the graph of $q: Z_2 \to X$. Let $\textrm{id}_X \times r: X \times \mathbb{P}(E_{2,\mathcal{L}}) \to X \times \sigma_2(X)$ be the product morphism. Consider the following diagram.

    \begin{equation}  \label{diagram of gamma}
    \xymatrix@=20pt{
        \Gamma_q \ar[r] \ar[d] & X \times Z_2 \ar[r] \ar[d] & X \times \mathbb{P}(E_{2, \mathcal{L}}) \ar[d]^{\textrm{id}_X \times r} \\
        \Delta(X) \ar[r] & X \times X \ar[r] &X \times \sigma_2(X)}
    \end{equation}

     The right-side box represents a pull back square, as $Z_2$ is the scheme-theoretic inverse image of $X$ under $r$. The left-side box is also a pull back square due to the property of the graph.

      Even if we find a nonsingular birational model $(\Gamma_q, X \times \mathbb{P}(E_{2,
 \mathcal{L}}))$ for the original pair $(\Delta(X), X \times \sigma_2(X))$ we still require a comparison theorem for two Segre classes of pairs. The following proposition establishes a strong connection between Segre classes.

 \begin{prop} [{\cite[Proposition 4.2 (a)]{fulton2013}}] Let $f : Y^{'} \to Y$ be a morphism of pure-dimensional schemes, $X \subset Y$ a closed subscheme, $X^{'} = f^{-1}(X)$ the inverse image scheme, $g : X^{'} \to X$ the induced morphism.

 If $f$ is proper, $Y$ irreducible, and $f$ maps each irreducible component of $Y^{'}$ onto $Y$, then 
$$ g_*(s(X^{'}, Y^{'})) = \deg(Y^{'}/Y) \; s(X, Y). $$ \end{prop}

     According to \cite[Proposition 4.2 (a)]{fulton2013}, we have 
\begin{equation} \label{eq main}
(\textrm{id}_X \times r)_* s( \Gamma_q, X \times \mathbb{P}(E_{2,\mathcal{L}}) ) = s(\Delta(X), X \times \sigma_2(X)).
\end{equation}

 Equation (\ref{eq main}) can be used to calculate (\ref{eq3}). Consider the following closed immersions:
 \begin{equation*}
     \Gamma_q \subset X \times Z_2 \subset X \times \mathbb{P}(E_{2,\mathcal{L}}).
 \end{equation*}

 Since each term is nonsingular, following three closed immersions are regular immersions. 
 \begin{align*}
 \Gamma_q &\subset X \times Z_2 \\
 \Gamma_q &\subset X \times \mathbb{P}(E_{2, \mathcal{L}}) \\
 X \times Z_2 &\subset X \times \mathbb{P}(E_{2, \mathcal{L}})
 \end{align*}

 Thus we have the following exact sequence of normal bundles: 
 \begin{equation} \label{eq5}
     0 \to N_{\Gamma_q/X \times Z_2} \to N_{\Gamma_q / X \times \mathbb{P}(E_{2,\mathcal{L}})} \to N_{X \times Z_2 / X \times \mathbb{P}(E_{2,\mathcal{L}})}|_{\Gamma_q} \to 0.
 \end{equation}
 Let $p : \Gamma_q \to Z_2$ be an isomorphism induced in \ref{diagram of gamma}. After simplification, we obtain the following:
\begin{align*}
  s(N_{\Gamma_q / X \times Z_2}) &= (\textrm{id} \times r|_{\Gamma_q})^*s(T_{\Delta(X)})\\
  s(N_{X \times Z_2 / X \times \mathbb{P}(E_{2,\mathcal{L}})}|_{\Gamma_q}) &= p^* s(N_{Z_2 / \mathbb{P}(E_{2,\mathcal{L}})}).
\end{align*}
 Using (\ref{eq5}) and applying Whitney's formula, we get:
 \begin{equation} \label{eq9}
     s(\Gamma_q, X \times \mathbb{P}(E_{2,\mathcal{L}})) = (\textrm{id} \times r|_{\Gamma_q})^*s(T_{\Delta(X)}) \cdot p^* s(Z_2, \mathbb{P}(E_{2,\mathcal{L}})).
 \end{equation}

 Consider the following pull back square of morphisms between varieties :

\begin{equation*}
    \xymatrix@=20pt{
 \Gamma_q \ar[r]^p \ar[d] & Z_2 \ar[d]^{r|_{Z_2}} \\
     \Delta(X) \ar[r]^{\Delta^{-1}} & X}
\end{equation*}

where $\Delta^{-1} : \Delta(X) \to X$ is an inverse of the diagonal embedding of $X$.
 By flat base change on Chow groups, we get $$ (\textrm{id} \times r|_{\Gamma_q})_* \circ p^* s(Z_2, \mathbb{P}(E_{2,\mathcal{L}})) = (\Delta^{-1})^* \circ r_* s(Z_2, \mathbb{P}(E_{2,\mathcal{L}})).$$ Note that $s(T_{\Delta(X)}) \cong (\Delta^{-1})^*s(T_X)$. From (\ref{eq main}) and (\ref{eq9}), we get the following:
\begin{equation} \label{eq10}
 s(\Delta(X), X \times \sigma_2(X)) = (\Delta^{-1})^* ( s(T_X) \cap r_*s(Z_2, \mathbb{P}(E_{2,\mathcal{L}}))).
\end{equation}

 Again by \cite[Proposition 4.2]{fulton2013}, we get
 \begin{equation} \label{eq6}
     s(\Delta(X), X \times \sigma_2(X)) = (\Delta^{-1})^* ( s(T_X) \cap s(X, \sigma_2(X)).
 \end{equation}

 So, it remains to compute the total Segre class $s(N_{Z_2 / \mathbb{P}(E_{2,\mathcal{L}})})$. Since both $r^{-1}(X)$ and $\mathbb{P}(E_{2,\mathcal{L}})$ are nonsingular varieties, we can regard $r^{-1}(X)$ as an effective Cartier divisor of $\mathbb{P}(E_{2,\mathcal{L}})$. So we can compute $s( Z_2,\mathbb{P}(E_{2,\mathcal{L}}) ) $ by \cite[Corollary 4.2.2]{fulton2013} as follow:
 \begin{equation} \label{eq7}
     s( Z_2,\mathbb{P}(E_{2,\mathcal{L}})) =  \frac{[Z_2]}{1 + [Z_2]}  .
 \end{equation}

 In order to proceed, it is necessary to express the term $[Z_2]$ in terms of the tautological line bundle $\zeta$ of $\mathbb{P}(E_{2,\mathcal{L}})$ and $\pi^* \beta$, where $\beta$ is a divisor on $X^{[2]}$ where $\pi : \mathbb{P}(E_{2, \mathcal{L}}) \to X^{[2]}$ is the bundle projection map. (cf. \cite[Chapter 3.3]{fulton2013}) This is achieved by calculating the first Chern class of $E_{2,\mathcal{L}}$ in proposition \ref{prop 2}.

Let $h$ be the divisor corresponding to a line bundle $\mathcal{L}$ on $X$. We denote the pullback of $h$ under the $i$-th projection by $h_i := \textrm{pr}_i^* h$. Since the morphism $\rho : Z_2 \to X^{[2]}$ is an involution map, we have 
$$\rho_* \eta^* h_1 = \rho_* \eta^* h_2.$$ Since $\mathbb{P}(E_{2,\mathcal{L}})$ is a nonsingular variety, we define $H = \rho_* \eta^* h_1 = \rho_* \eta^* h_2$. Let $\partial X^{[2]}$ be the exceptional divisor of Hilbert-Chow morphism $X^{[2]} \to X^{(2)}$. We can represent $\partial X^{[2]}$ as $2\delta$ where $\delta$ is a Cartier divisor on $X^{[2]}$. Note that $\delta = \frac{1}{2}\rho_* E$ as Cartier divisors.

 \begin{prop} \label{prop 2}
     $c_1(E_{2,\mathcal{L}}) = H - \delta$ as a cycle of $A^1(X^{[2]})_{\mathbb{Q}}$.
 \end{prop} 
\begin{proof}
Let $\pi : \mathbb{P}(E_{2,\mathcal{L}}) \to X^{[2]}$ be the projection map of a projective bundle. 
    Consider the normal bundle $\mathcal{N}:=N_{Z_2 / X\times X^{[2]}}$ and the closed immersion $j: Z_2 \to X \times X^{[2]}$ with the composition $q: Z_2 \to X$ of $\eta$ and the first projection. The morphism $\pi_1|_{Z_2} : Z_2 \to X^{[2]}$ is a finite flat morphism, and $\pi_{2 \; *} (\pi_1^* \mathcal{L} \otimes \mathcal{O}_{Z_2} )$ is a locally free sheaf on $X^{[2]}$, so by Grauert's theorem again, all higher direct images $R^i\pi_{2 \; *} (\pi_1^* \mathcal{L} \otimes \mathcal{O}_{Z_2} )$ vanish for $i \geq 1$. Let $T_X$ be the tangent bundle of $X$.

Since $j$ is an affine morphism, there are no higher direct images and we obtain followings :
 $$\textrm{ch} \; E_{2,\mathcal{L}} = \pi_{2 \; *} ( \textrm{ch} \; (j_* \mathcal{O}_{Z_2}) \cdot \pi_1^* \textrm{ch} \; \mathcal{L} \cdot \pi_1^* \textrm{td} \; T_X)$$ and $$\textrm{ch} \; j_*\mathcal{O}_{Z_2} = j_*((\textrm{td} \; \mathcal{N})^{-1})$$ 
 by applying the Grothendieck-Riemann-Roch theorem. Note that $\pi_2 \circ j = \rho$.

By the projection formula, we have:
\begin{align*} 
    \textrm{ch} \; E_{2,\mathcal{L}} &= \pi_{2 \; *} ( \textrm{ch} \; (j_* \mathcal{O}_{Z_2}) \cdot \pi_1^* \textrm{ch} \; \mathcal{L} \cdot \pi_1^* \textrm{td} \; T_X)  \\ 
     &= \pi_{2 \; *} ( j_* (\textrm{td} \; \mathcal{N} )^{-1} \cdot \pi_1^* \textrm{ch} \; \mathcal{L} \cdot \pi_1^* \textrm{td} \; T_X) \\
     &= (\pi_{2 \; *} \circ j_*)( (\textrm{td} \; \mathcal{N})^{-1} \cdot j^*\pi_1^* \textrm{ch} \; \mathcal{L} \cdot j^* \pi_1^* \textrm{td} \; T_X) \\ 
     &= \rho_*( (\textrm{td} \; \mathcal{N})^{-1} \cdot   q^* \textrm{ch} \; \mathcal{L} \cdot q^* \textrm{td} \; T_X ).
\end{align*}

Recall that $E$ is the exceptional divisor of $\eta : Z_2 \cong Bl_{\Delta(X)} X \times X \to X \times X$. Consider the exact sequence $$ 0 \to \mathcal{N}^* \to q^* \Omega_X \to \mathcal{O}_E(-E) \to 0 $$ as given in \cite[Lemma 2.1]{GF93}. (Note that the morphism $q$ in proposition \ref{prop 2}  and in \cite{GF93} are different morphisms.)  
 Taking dual of this sequence, we obtain:
 $$ 0 \to q^* T_X \to \mathcal{N} \to \mathcal{E}xt^1(\mathcal{O}_E(-E), \mathcal{O}_{Z_2}) \to 0.$$

 Therefore we have $$(\textrm{td} \; \mathcal{N})^{-1} = (q^* \textrm{td} \; T_X)^{-1} \cdot (\textrm{td} \;  \mathcal{E}xt^1(\mathcal{O}_E(-E), \mathcal{O}_{Z_2}))^{-1}.$$ 
 
 To evaluate $ \textrm{td} \; \mathcal{E}xt^1(\mathcal{O}_E(-E), \mathcal{O}_{Z_2}) $, we consider the exact sequence: $$0 \to \mathcal{O}_{Z_2}(-2E) \to \mathcal{O}_{Z_2}(-E) \to \mathcal{O}_E(-E) \to 0.$$

Taking the dual of this sequence, we have:
$$ 0 \to \mathcal{O}_{Z_2}(E) \to \mathcal{O}_{Z_2}(2E) \to \mathcal{E}xt^1(\mathcal{O}_E(-E), \mathcal{O}_{Z_2}) \to 0.$$

 Therefore, we obtain 
 $$ \textrm{td} \; \mathcal{E}xt^1(\mathcal{O}_E(-E), \mathcal{O}_{Z_2}) = \frac{2E}{1 - e^{-2E}} / \frac{E}{1 - e^{-E}} = \frac{2}{1 + e^{-E}}$$
 and  
 \begin{align*}
     \textrm{ch} \; E_{2,\mathcal{L}} &= \rho_* \left(q^*\textrm{ch} \; \mathcal{L} \cdot \frac{1}{2} (1+e^{-E}) \right) \\
                                 &= \rho_*\left( \left(1 + \eta^* h_1 + \frac{1}{2!} \eta^* h_1^2 + \frac{1}{3!} \eta^* h_1^3 + \cdots \right)\cdot \left( 1 - \frac{1}{2}E + \frac{1}{4} E^2 - \frac{1}{12} E^3 + \cdots \right) \right) \\
                                 &= \rho_* \left( 1 + \left(\eta^*h_1 - \frac{1}{2} E \right) + \left(\frac{1}{2}\eta^*h_1^2 - \frac{1}{2}\eta^*h_1 \cdot E + \frac{1}{4} E^2 \right) + \cdots  \right) \\
                                 &= 2 + (H - \delta ) + \left( \frac{1}{2} \rho_*\eta^*h_1^2- \frac{1}{2}\rho_*(\eta^*h_1 \cdot E) + \frac{1}{4}\rho_* E^2 \right) + \cdots.
 \end{align*}
  So we obtain $c_1(E_{2,\mathcal{L}}) = H-\delta$.  \end{proof}

 In equation (\ref{eq7}), it is deduced that 
 \begin{equation} \label{eq11}
  s(Z_2, \mathbb{P}(E_{2,\mathcal{L}})) = \sum_{i \geq 0} (-1)^i [Z_2]^{i+1} =\sum_{i \geq 0} (-1)^i ([Z_2]|_{Z_2})^i.
 \end{equation}
To continue, we require an expression of $[Z_2]|_{Z_2}$ as a linear combination of generators of the Chow ring $A^*Z_2$. 

 \begin{prop} \label{prop 3}
      $[Z_2]|_{Z_2} = 2E + \eta^*(h_1 - h_2)$ in $A^1 (Z_2)_{\mathbb{Q}}$ where $[Z_2]$ is a cycle associated to scheme $Z_2$.
 \end{prop} 
 
  \begin{proof}
      Recall that $E$ is the inverse of $\Delta(X)$ under the blow up $\eta$ and $\delta$ is a half of Cartier divisor associated to the boundary component of $X^{[2]}$ such that $\delta = \frac{1}{2}\rho_*E$. 
 We denote the canonical divisors of $X^{[2]}$ and $Z_2$ as $K_{X^{[2]}}$ and $K_{Z_2}$, respectively. The ramification divisor of $\rho$ is $E$, so we have
  \begin{equation} \label{eq8}
      \rho^* K_{X^{[2]}} = K_{Z_2} - E .
  \end{equation}
   Equation (\ref{eq8}) can be found in some literature, or it can be deduced directly from the following exact sequence of sheaves:
   \begin{equation*}
       0 \to \rho^* T_{X^{[2]}}^{*} \to T_{Z_2}^{*} \to \mathcal{O}_E(-E) \to 0.
   \end{equation*}

  Let $\zeta$ be the first Chern class of the tautological line bundle of $\mathbb{P}(E_{2,\mathcal{L}})$. 
 The normal bundle $N_{Z_2 / \mathbb{P}(E_{2,\mathcal{L}})} $ is represented by $\mathcal{O}_{Z_2}([Z_2])$, $$c_1(\mathcal{O}_{Z_2}([Z_2])) = K_{Z_2} - K_{\mathbb{P}(E_{2,\mathcal{L}})}|_{Z_2}.$$ In addition, we have the following adjunction formula for the projective bundle:
 $$ K_{\mathbb{P}(E_{2,\mathcal{L}})} = -2\zeta + \pi^* c_1(E_{2,\mathcal{L}}) + \pi^* K_{X^{[2]}}. $$

 Note that $(\pi^* K_{X^{[2]}})|_{Z_2} = \rho^* K_{X^{[2]}} = K_{Z_2} -E = K_{Z_2} -\rho^* \delta$. Then $$c_1(\mathcal{O}_{Z_2}([Z_2])) = 2\zeta|_{Z_2} -\rho^*(H-\delta) + \rho^*\delta$$ by proposition \ref{prop 2}. Recall that $c_1(\mathcal{O}_{Z_2}([Z_2])) = [Z_2]|_{Z_2}$ and hence we get 
 \begin{equation} \label{eq15}
      [Z_2]|_{Z_2} = 2\zeta|_{Z_2} - \rho^*(H -2\delta).
 \end{equation}

 We need information about $\zeta|_{Z_2}$. We know that $Z_2$ is a closed subvariety of $$X \times X^{[2]} \subset \mathbb{P}H^0(X, \mathcal{L}) \times X^{[2]}.$$ Therefore, we can choose a suitable section of $\zeta|_{Z_2}$ whose zero locus is $ \pi_1^*h \cap Z_2$ where $\pi_1 : X \times X^{[2]} \to X$ is a projection map.

Recall that we have an isomorphism $Z_2 \cong Bl_{\Delta(X)} (X \times X)$, which allows us to view $$Z_2 \subset X \times X^{[2]} \to X$$ as equivalent to $$Bl_{\Delta(X)} (X \times X) \to X \times X \to X$$ where the second morphism is $\textrm{pr}_1$. (cf. \cite[Section 1]{ullery}) Thus, we obtain $\zeta|_{Z_2} = \eta^*h_1$.

Substituting this result into (\ref{eq15}), we find that $$[Z_2]|_{Z_2} = 2E + \eta^*(h_1 - h_2).$$
\end{proof}

\section{Main theorem}
 In order to derive the degree formula for the 3-secant variety, it is necessary to utilize each term of \cite[Theorem 8.2.8]{joins1999}. To accomplish this, we present the following lemma.

\begin{lem} \label{lem}
 Let $X$ be a nonsingular projective variety that is embedded by a 5-very ample line bundle. Let $Y$ be the 2-secant variety $\sigma_2(X)$. Let $J$ be the ruled join $J(X, Y)$. Then $\deg(J/XY)$ is 3.
\end{lem}

\begin{proof} Let $w$ be a general point of $\sigma_3(X) - \sigma_2(X)$. If there are two secant planes that contain $w$, 6 points of $X$ do not satisfy the independent condition of 5-very ampleness.(cf. \cite[Remark 1.7]{cattaneo2020}) So, there exists a unique plane spanned by three points of $X$ that contains $w$. Let $x, y,$ and $z$ be three distinct points of $X$ such that their linear span contains $w$. Let $a, b,$ and $c$ be the points of intersection of $\overline{xw}$ with $\overline{yz}$, $\overline{yw}$ with $\overline{zx}$, and $\overline{zw}$ with $\overline{xy}$, respectively. Then, the three points of the ruled join $J(X, Y)$ corresponding to the ratios between ${(x,w), (w,a)}, {(y,w), (w,b) }$, and ${(z,w), (w,c)}$ are exactly the inverse image of the rational map $J(X, Y) \dashrightarrow XY$.
\end{proof}

  Note that $s_k(C_{\Delta(X)}(X \times X)) = s_{n-k}(T_X)$ as a Segre class of a locally free sheaf. With this notation established, we now proceed to the proof of the main theorem:

 \begin{proof}[Proof of main theorem] 
     
 Note that $X$ is a smooth projective variety of dimension $n$, and $E \subset Z_2$ is isomorphic to a projective bundle associated with the tangent bundle : $\mathbb{P}(T_X^*) \cong P(C_{\Delta(X)}(X \times X))$ over $X$. Recall that $h$ is a Cartier divisor associated with $\mathcal{L}$ on $X$, and $h_i$ a Cartier divisor on $X \times X$ defined by $ h_i = \textrm{pr}_i^* \; h$.  Using equations (\ref{eq6}), (\ref{eq11}), and proposition \ref{prop 3}, we obtain:
  \begin{align} \label{eq12}
      s(X, \sigma_2(X)) &= q_* \left( \sum_{i \geq 0} (-1)^i (2E + \eta^*(h_1-h_2))^i \right).
  \end{align}

  Let $g : P(C_{\Delta(X)}(X \times X)) \to X$ be the projection map. Let $\mathcal{O}(1)$ be the tautological line bundle of $P(C_{\Delta(X)}(X \times X))$. Since $E$ and $P(C_{\Delta(X)}(X \times X))$ are isomorphic, $$ E|_{E} = c_1(\mathcal{O}(-1))$$ holds.

  Therefore we have 
  $$\eta_* E^i = (-1)^{i-1} \eta_* (-E|_E)^{i-1} = (-1)^{i-1} s_{i-n}(T_{X}) = (-1)^{i-1}s_{2n-i}(C_{\Delta(X)}(X \times X)) $$ for $i \geq n$ by (\ref{segre tangent bd normal cone}). Since $\eta^*(h_1-h_2) \cdot E$ is zero, we obtain that 
  \begin{equation} \label{eq13}
   \eta_* \sum_{l=n}^{2n} (-1)^l (2E + \eta^*(h_1-h_2))^l = \sum_{l=n}^{2n} (-1)^l(h_1-h_2)^l - \sum_{l=n}^{2n} 2^l s_{l-n}(T_X) \cap [X] . 
\end{equation}
  
 The remaining calculation simply involves taking the first projections, which yields: 
 \begin{align*}
     (\textrm{pr}_1)_* (h_1-h_2)^l &= \sum_{a = 0}^{l} (-1)^{l-a} { l \choose a } h^a \; (\textrm{pr}_1)_* \textrm{pr}_2^* \; h^{l-a} \\
                                   &= (-1)^n \;d\; {l \choose l-n} h^{l-n}.  \qquad (l \geq n)
 \end{align*}

 By substituting (\ref{eq13}) into (\ref{eq12}) we have:
 \begin{align*}
     s(X, \sigma_2(X)) &= \sum_{l=n}^{2n} (-1)^{n+l} \;d\; {l \choose n} \;h^{l-n}  - \sum_{l=n}^{2n} 2^l s_{l-n}(T_X) \cap [X] \\
                       &= \sum_{i=0}^{n} (-1)^i \; d\; { i+n \choose n} \; h^i - \sum_{i=0}^{n}2^{i+n}s_i(T_X) \cap [X] .
 \end{align*}

From equation (\ref{eq10}), we obtain:
 \begin{align*}
s( \Delta(X), X \times \sigma_2(X)) &=  s(T_X) \cap s(X, \sigma_2(X)) \\
                                    &= \sum_{i=0}^{n} \sum_{a+b=i} \{ (-1)^b \;d\;{ b+n \choose n} \;h^b \cdot s_a(T_X) \cap [X] - 2^{b+n}s_a(T_X) \;s_b(T_X) \cap [X] \}.
      \end{align*}
 by omitting the pull back $(\Delta(X) \to X)^*$. 
 Therefore, we obtain the main formula from equation (\ref{eq2}), which completes the proof. \end{proof}

 \begin{rem} Recall that we use the convention for projective bundles and tautological line bundles as in \cite[Appendix B.5.5]{fulton2013}. \end{rem}

\begin{cor}[cf. {\cite[Proposition 6.1]{ChouSong18}} and {\cite[Chapter 4.3]{fulton2013}}]
The multiplicity of $\sigma_2(X)$ along $X$ is $d - 2^n$.
\end{cor}

\begin{proof}
    The multiplicity of $\sigma_2(X)$ along $X$ is the coefficient of $[X]$ in the class $s(X, \sigma_2(X))$.
\end{proof}

\section{Examples}
\begin{Ex} [In the case of curves]
  Consider the case where $C$ is a smooth projective curve. Let $g$ be the genus of a curve $C$. The degree of each term of Segre class $T_C$ is given by $\deg s_0(T_C) = d$ and $\deg s_1(T_C) = 2g-2$. By substituting those terms in the main formula we get: 
  $$ \deg \sigma_3(C) = \frac{1}{3!} (d^3 - 9d^2 +26d +24 -6dg -24g)  $$
  which can be found in \cite[Proposition 1]{soule04}. \end{Ex}

\begin{Ex} [In the case of surfaces]
   Consider the case where $S$ is a smooth projective surface. Let $K$ be the canonical divisor of $S$, let $d = h^2$, let $\pi = h \cdot K$, let $\kappa = K^2$, and let $e = c_2$ be the topological Euler characteristic.
The total Segre class of the tangent sheaf $T_S$ of $S$ has degree 0, 1, and 2 terms, which are given by:
 $$ s_0(T_S) = [S], \; s_1(T_S) = K, \; s_2(T_S) = \kappa - e.$$
 By substituting those terms in the main formula we get:
 $$ \deg \sigma_3(S) = \frac{1}{3!}(d^3 -30d^2 + 224d -3d(5\pi + \kappa -e) + 192\pi + 56\kappa -40e) $$
 which can be found in \cite[Section 4] {lehn}.

 We can directly calculate the total Segre classes $s(S, \sigma_2(S))$. Note that $[Z_2]|_{Z_2}$ is represented by $ 2E + \eta^*(h_1 - h_2)$. This yields the followings:
 
\begin{align*} 
    ([Z_2]|_{Z_2})^2 &= 4E^2 + 4E \; \eta^*(h_1-h_2) + \eta^*(h_1-h_2)^2  \\ 
    ([Z_2]|_{Z_2})^3 &= 8E^3 + 12E^2 \; \eta^*(h_1-h_2) +6E \; \eta^*(h_1-h_2)^2 +\eta^*(h_1-h_2)^3 \\
    ([Z_2]|_{Z_2})^4 &= 16E^4 + 32E^3 \; \eta^*(h_1-h_2) + 24E^2 \; \eta^*(h_1-h_2)^2 + 8E \; \eta^*(h_1-h_2)^3 + \eta^*(h_1-h_2)^4
\end{align*}

 Since $\eta_* E^l = (-1)^{l-1} \; s_{l-2}(T_{\Delta(S)})$ and $\Delta(S) \cdot ( h_1-h_2) = 0$, we obtain:

 \begin{align*}
\eta_* ([Z_2]|_{Z_2})^2 &= -4[\Delta(S)] + (h_1-h_2)^2           &  q_* ([Z_2]|_{Z_2})^2 &= (d-4)[S]           \\
\eta_* ([Z_2]|_{Z_2})^3 &= 8s_1(T_{\Delta(S)}) + (h_1-h_2)^3     &  q_* ([Z_2]|_{Z_2})^3 &= 8s_1(T_S) + 3dh    \\
\eta_* ([Z_2]|_{Z_2})^4 &= -16s_2(T_{\Delta(S)}) + 6h_1^2h_2^2   &  q_* ([Z_2]|_{Z_2})^4 &= -16s_2(T_S) + 6d^2 .\\
 \end{align*}

The total Segre class $s(S, \sigma_2(S))$ is the summation of the three terms on the right-hand side of the equation. Therefore, the total Segre class $s( \Delta(S), S \times \sigma_2(S))$ is given by:
$$ (d-4)[S] + ((d-12)s_1(T_S) -3dh) + (-8s_1(T_S)^2 + (d-20)s_2(T_S) -3dh \cdot s_1(T_S) + 6d^2).   $$
 Using \cite[Theorem 8.2.8]{joins1999}, we obtain the correct formula for $\deg \sigma_3(S)$ again. \end{Ex}

\end{document}